\documentclass[12pt]{amsart}
\usepackage{amsmath,amsthm,amsfonts,amssymb,eepic}
\usepackage{hyperref}
\usepackage{url} 
\usepackage{multicol}
\usepackage{graphicx}
\usepackage{pgf}

\theoremstyle{plain}

\newtheorem{corollary}{Corollary}[section]
\newtheorem{definition}{Definition}[section]

\newtheorem{lemma}{Lemma}[section]
\newtheorem{proposition}{Proposition}[section]
\newtheorem{remark}{Remark}[section]
\newtheorem{theorem}{Theorem}[section]

\numberwithin{equation}{section}

\title[Bond Percolation on three fractals]{Bond percolation on a non-p.c.f. Sierpi\'nski Gasket, iterated barycentric subdivision of a triangle, and Hexacarpet}
\author[D.~Lougee]{Derek Lougee}
\address[D.~Lougee]{Department of Mathematics, Cornell University, Ithaca, NY, 14853-4201\\ Mailing address: Department of Economicc, University of Carlos III, Madrid Spain}

\author[B.~Steinhurst]{Benjamin Steinhurst}
\address[B.~Steinhurst]{Department of Mathematics, Cornell University, Ithaca, NY, 14853-4201\newline Mailing address: Department of Mathematics and Computer Science, McDaniel College, Westminster MD, 21157}
\email[B.~Steinhurst]{bsteinhurst@mcdaniel.edu}

\date{\today}

\begin{document}

\begin{abstract}
We investigate bond percolation on the iterated barycentric subdivision of a triangle, the hexacarpet, and the non-p.c.f. Sierpinski gasket. With the use of known results on the diamond fractal, we are able to bound the critical probability of bond percolation on the non-p.c.f. gasket and the iterated barycentric subdivision of a triangle from above by 0.282. We then show how both the gasket and hexacarpet fractals are related via the iterated barycentric subdivisions of a triangle: the two spaces exhibit duality properties although they are not themselves dual graphs. Finally we show the existence of a non-trivial phase transition on all three graphs.

MSC2012: 60K35, 28A80, 52M20

Keywords: Percolation, barycentric subdivision, fractals.
\end{abstract}

\maketitle

\section{Introduction}\label{sec:intro}


A bond percolation cluster in an ambient graph is a random subgraph whose edges are chosen from the ambient graph's edges independently with some probability $p$. The critical probability of bond percolation is the supremum of those $p$ such that the largest component of the cluster is almost surely finite, or conversely the infemum of the $p$ so that the largest cluster is almost surely infinite. The study of percolation began with Broadbent and Hammersley in 1957 \cite{Hammersley} with a major contribution coming from Kesten in 1980 \cite{Kesten} who proved rigorously that the previously calculated value of the critical probability of bond percolation on $\mathbb{Z}^{2}$ was correct. A more detailed history and overview of the topic can be found in Grimmett \cite{Grimmett-PercBook}. The original idea was to use percolation as a model for the formation of random clusters in a homogenous medium, usually $\mathbb{Z}^{d}$, for example the Ising model where an electron with either up or down spin is placed at each vertex of the graph.


As ambient graphs, lattices and Cayley graphs of finitely generated groups have proven to be quite tractable for proving the existence and upper bounds on the critical probability of bond percolation. It was a natural progression to consider the graph approximations to common fractals like the Sierpinski gasket \cite{Shinoda-preSG} and Sierpinski carpets \cite{Kum1997,Shinoda-carpets,Shinoda-existence}. The main observation in these studies was that the graphs representing the gasket and the carpet could be viewed as subgraphs of $\mathbb{Z}^{2}$ so many methods could be adapted. The difference between the gasket and carpet cases is that for the gasket the critical probability is trivially equal to one. For the carpet case the results are not-trivial. A departure from this was the work in \cite{HamblyKumagai} on the diamond fractal whose graph approximations have unbounded vertex degree cannot all be subgraphs of $\mathbb{Z}^{d}$ for any fixed $d$. The non-post-critically finite (non-p.c.f.) Sierpinski gasket, like the diamond fractal, does not have approximating graphs that can be embedded in a single finite dimensional lattice but as we shall see bears enough resemblance to the diamond to be amenable to the same methods as in \cite{HamblyKumagai}.


A self-similar fractal has a geometric structure that allows it to be described as being the union of scaled copies of itself, with contraction mappings $\phi_i$ mapping the fractal into itself. For example, the unit interval is the union of $[0,\frac12]$ and $[\frac12,1]$ with mappings $\pi_0(x) = \frac{1}{2}x$ and $\phi_1(x) = \frac{1}{2}x+\frac{1}{2}$. A cell structure can then be iteratively defined. The fractal $F$, is itself a level $0$ cell. Then $\phi_i(F)$ are the level one cells, and $\phi_i(\phi_j(F))$ the level two cells, and so on. We say that a fractal is finitely ramified if $F$ can be separated into disjoint components by removing a finite number of points from $F$. In our example, the unit interval is disconnected by removing the point $\{\frac{1}{2}\}$. For a finitely ramified fractal we can ask how many cells of level $n$ can overlap at a given point? In the example of the unit interval only two cells of a given level can ever overlap at a point. \emph{Post-critically finite} is a more delicate condition to state and the distinction is not critical to this paper so the interested reader can look to \cite{Kigami} for the precise definition. It will suffice to state that being not finitely ramified will imply being not post-critically finite. 
Consider Figure \ref{fig:nonpcf}, as each triangular face is replaced by a copy of the level one approximation to the non-p.c.f. Sierpinski gasket the number of points that must be removed to separate the succeeding approximations is growing without bound. So the non-p.c.f. Sierpinski gasket is not finitely ramified and so not post-critically finite.

\begin{theorem}\label{thm:main}
Let $S$ denote the non-p.c.f. Sierpinski gasket and let $H$ denote the hexacarpet. Then $0 < p_S < 0.282$ and $0.718 < p_H < 1$.
\end{theorem}

\begin{proof}
Denote by $T$ the iterated barycentric subdivision of the triangle and $p_G$ is the critical threshold probability for bond percolation on graph $G$. By Theorems \ref{thm:pcT} and \ref{thm:pSup} we have that both $p_S$ and $p_T$ are bounded above by $0.282$. From Corollary \ref{cor:pHup} we have that $p_H < 1$. Theorem \ref{duality} gives $1-p_H = p_T$ so $p_T>0$. Finally from Corollary \ref{cor:SScomp} we have that $p_T \le 2 p_S$ from which the lower bound for $p_S$ follows. 
\end{proof}

Sections \ref{sec:perc} and \ref{sec:diamond} introduce the bond percolation problem and diamond fractals, setting out the necessary notions and notation that will be used in the rest of the paper. Section \ref{sec:BSSG} defines the barycentric subdivision operation and shows that on the iterated barycentric subdivision of a triangle that the diamond fractal can be used to obtain the $0.282$ upper bound used in Theorem \ref{thm:main}. This argument is extended to the non-p.c.f. Sierpinski gasket and the section ends by comparing percolation on the barycentric subdivisions of a triangle and the non-p.c.f Sierpinski gasket. In Section \ref{sec:hex} the hexacarpet is introduced as the limit of the dual graphs of the barycentric subdivisions of a triangle and it is shown that the critical threshold probability for bond percolation on the hexacarpet is strictly less than one. Finally in Section \ref{sec:dual} the relationship between the critical thresholds on the hexacarpet and the iterated barycentric subdivision of a triangle is discussed.

\section*{Acknowledgements}
We first thank Luke Rogers for first suggesting that we look at the non-p.c.f. Sierpinski gasket, Takeshi Kumagai for suggesting the thinning procedure, and Joe Chen for his insight into the isoperimetric problem. We also thank Dan Kelleher and Alexander Teplyaev for bringing the hexacarpet to our attention. Benjamin Hambly make several useful comments himself and solicited useful comments from another reader. Finally the Cornell University Department of Mathematics for supporting both authors.

\section{Bond Percolation}\label{sec:perc}
The standard model for bond percolation is to have a fixed finite or infinite graph with its edge and vertex sets; in the case of fractals which are defined as scaling limits of a sequence of graphs the notion of bond percolation needs to be expanded. In \cite{Shinoda-preSG,Shinoda-existence,Shinoda-carpets} there are natural unbounded infinite graphs with no vertex accumulation points which are the increasing limit of finite graphs. Since bond percolation does not use the ``lengths'' of the edges and there is no graph theoretic difference between the compact limit of scaled graphs and the unbounded limit of unscaled graphs; this provides a good notion of percolation through such a fractal space. On the other hand Hambly and Kumagai in \cite{HamblyKumagai} consider the diamond fractal as a self-similarly arranged graph with infinite vertex degrees and adopt an iterative notion of percolation which will be discussed in detail in Section \ref{sec:diamond}. 

The basic reference for percolation theory used in this paper is \cite{Grimmett-PercBook}. Following the notation there we use a model of bond percolation that makes varying the parameter $p$ simple. Given a graph $G=(V,E)$ assign i.i.d. uniform random variables on $[0,1]$ to the edges in $E$. 
Let $\Omega_G = [0,1]^{E}$ and $\mathbb{P}_G = \prod_{e \in E} \lambda_{[0,1]}$, where $\lambda_{[0,1]}$ is Lebesgue measure on the $[0,1]$. Define $E_p = \{ e \in E :\ \omega(e) < p \}$, if $e \in E_p$ we say that $e$ is an open edge. Let $G_p = (V,E_p)$ be the random graph consisting of the original vertex set and the open edges. If the graph $G$ has a marked vertex called the origin the open connected component of the graph $G_p$ containing the origin is called $\mathcal{O}_{G_p}$. Site percolation is the complementary process where the vertex set is randomized rather than the edge set. 

\begin{definition}
The percolation probability is defined as
\begin{equation*}
	\theta_G(p) = \mathbb{P}_G(|\mathcal{O}_{G_p}| = \infty).
\end{equation*}
Where $|\mathcal{O}_{G_p}|$ is the number of vertices in $\mathcal{O}_{G_p}$.
\end{definition}

We will, when context allows, drop the dependence of $\mathbb{P}_G$ on $G$ from the notation and instead reference it in $\theta_G$ which will always be understood to depend on the relevant measure for the underlying graph. 

The existence of an infinite open cluster not necessarily containing the origin is actually a $0-1$ event. In fact, an infinite open cluster exists with probability one if and only if there is a positive probability that the origin is in an infinite cluster \cite{Grimmett-PercBook}. 

\begin{remark}
If $G = \mathbb{Z}^{d}$, the event $\{|\mathcal{O}_{G_p}| = \infty\}$ can be interpreted as `there exists an open cluster that ``crosses'' from the origin in $\mathbb{Z}^{d}$ to infinity with positive probability.' This observation is the basis for the percolation model discussed in the next section.
\end{remark}

\begin{remark}\label{rem:mono1}
The function $\theta_G(p)$ is nondecreasing in $p$. If the environment is sampled, putting i.i.d. uniform labels on each edge, then by choosing those edges with $\omega(e) < p + \epsilon$ instead of $\omega(e)<p$ to be open then $|\mathcal{O}_{G_p}| = \infty$ implies that $\mathcal{O}_{G_{p+\epsilon}}$ has the same edges as $\mathcal{O}_{G_p}$ and potentially more so the connected component containing the origin is infinite. For a similar reason $\theta_{G}(p)$ is nondecreasing in the graph $G$. That is given a subgraph $H \subset G$, $\theta_{H}(p) \le \theta_{G}(p)$. 
\end{remark}

\begin{lemma}\label{lem:mono2}
If $H$ is a graph obtained from $G$ by declaring an equivalence relation on the vertices of $G$ such that each equivalence class contains a uniformly bounded finite number of vertices and taking the quotient of $G$ modulo this relation, then $\theta_G(p) \le \theta_H(p)$. 
\end{lemma}

\begin{proof}
Suppose $G_p$ had an infinite component containing the origin. Then the vertices of this component would have to be shared between infinitely many equivalence classes since each class contains only finitely many vertices. Furthermore, the connected component would stay connected through this construction. The edge sets of $H$ and $G$ are the same so $\mathbb{P}_H = \mathbb{P}_G$ so the events are comparable. A strict inequality is possible since if the origin were not in an infinite component of $G_p$ then the quotient might connect the origin to an infinite component.
\end{proof}

\begin{definition}
The critical probability of percolation on a given graph $G$ is given by 
\begin{equation*}
	p_G = \sup\{p :\ \theta_G(p) = 0\}.
\end{equation*}
Typically the critical probability is denoted $p_c$ when there is only a single ambient graph under consideration. With this convention in mind, in the next section a family of critical probabilities will be denoted $p_c(m,n)$ where the $m$ and $n$ indicate the ambient graph. 
\end{definition}

We can say there is a phase transition from a subcritical regime where open clusters are almost surely finite to a supercritical regime where there is almost surely an infinite cluster if $p_G \in (0,1)$. That is when $p_G$ is not zero or one. In the course of this paper we will consider several graphs each with its own critical probability which will need to be compared; to reduce confusion we will always indicate which graph is being considered in the notation.

\section{The Diamond Fractal}\label{sec:diamond}

\begin{figure}[t]
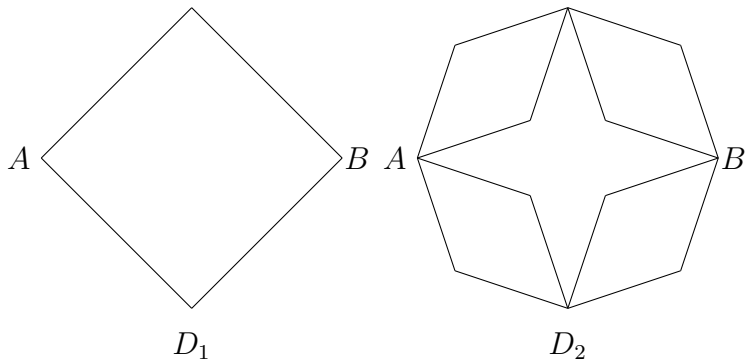

\begin{pgfpicture}{-0.3cm}{-0.5cm}{9.2cm}{4cm}
\pgfxyline(0,2)(2,4)
\pgfxyline(0,2)(2,0)
\pgfxyline(2,4)(4,2)
\pgfxyline(2,0)(4,2)
\pgfputat{\pgfxy(-0.3,2)}{\pgfbox[center,center]{$A$}}
\pgfputat{\pgfxy(4.2,2)}{\pgfbox[center,center]{$B$}}

\pgfxyline(5,2)(5.5,3.5)
\pgfxyline(5,2)(6.5,2.5)
\pgfxyline(5.5,3.5)(7,4)
\pgfxyline(6.5,2.5)(7,4)
\pgfxyline(5,2)(5.5,0.5)
\pgfxyline(5,2)(6.5,1.5)
\pgfxyline(5.5,0.5)(7,0)
\pgfxyline(6.5,1.5)(7,0)
\pgfxyline(7,4)(8.5,3.5)
\pgfxyline(7,4)(7.5,2.5)
\pgfxyline(8.5,3.5)(9,2)
\pgfxyline(7.5,2.5)(9,2)
\pgfxyline(7,0)(7.5,1.5)
\pgfxyline(7,0)(8.5,0.5)
\pgfxyline(7.5,1.5)(9,2)
\pgfxyline(8.5,0.5)(9,2)
\pgfputat{\pgfxy(4.7,2)}{\pgfbox[center,center]{$A$}}
\pgfputat{\pgfxy(9.2,2)}{\pgfbox[center,center]{$B$}}

\pgfputat{\pgfxy(2,-0.5)}{\pgfbox[center,center]{$D_1$}}
\pgfputat{\pgfxy(7,-0.5)}{\pgfbox[center,center]{$D_2$}}
\end{pgfpicture}
\caption{\label{Diamond} First and Second Graph Approximations of $D(2,2)$}
\end{figure}

Percolation on the diamond fractal has been studied by Hambly and Kumagai \cite{HamblyKumagai}. In this section we will reproduce some of their results and state a formula for the critical probability of a diamond fractal that is mentioned briefly in their concluding discussion.

To construct the diamond fractal $D(m,n)$, begin with a line segment $(A,B)$ denoted $D_0(m,n)$ where $m,n \ge 2$ denote the number of branches and subdivisions respectively of each edge to be made. The first approximation of the diamond fractal, $D_1(m,n)$  is generated by replacing the one edge in $D_0(m,n)$ by $m$ non-intersecting paths of $n$ edges. Figure \ref{Diamond} shows the first and second approximations to $D(2,2)$. We shall mostly be interested in $D(2,2)$ and $D(4,2^{k})$. 

The notion of percolation on diamond fractals that we use here is that given the probability $p$ of each edge being open, what is the probability of connecting one end of the fractal to the other?  It is straight forward to conclude by self-similarity that this is the same as the probability of the open cluster crossing a cell of the fractal at some level. The key to this observation is that in the limit any cell of any level is the same as the whole fractal as infinite graphs. To access this crossing probability we take the limit of the crossing probabilities for all approximating graphs $D_l(m,n)$ and take their limit. Let $p_c(D_l(m,n))$ be the probability of there being an open cluster in $(D_l(m,n),E_p)$ which contains both endpoints. 
Then the crossing probability of the diamond with parameter $p$ is $\lim_{l \rightarrow \infty} p_{c}(D_l(m,n))$. 
The crossing probability is the notion of percolation that we shall use for the non-p.c.f. Sierpinski gasket. On the diamond fractal crossing is the same as connecting the components of the boundary (the extreme left and right hand vertices). We will see that for the non-p.c.f. Sierpinski gasket these two notions to not coincide.

\begin{proposition}{\cite{HamblyKumagai}}\label{prop:level1}
Given $p \in [0,1]$ the probability of not being able to cross $D_1(m,n)$ is 
\begin{equation*}
	\mathbb{P}(\text{no crossing}) = (1-p^{n})^{m},
\end{equation*}
where a crossing is the event that there exists an open cluster containing both $A$ and $B$.
\end{proposition}

\begin{proof}
In order for there to be an open cluster containing $A$ and $B$ at least one of the $m$ branches must have all $n$ edges open. So each branch has a probability of not connecting $A$ and $B$ of $1-p^{n}$. For there to be no branch which is completely open all $m$ branches must not connect $A$ and $B$. Hence the probability of no crossing is $(1-p^{n})^{m}$. 
\end{proof}

\begin{proposition}\label{prop:fmn}
The functions $f_{m,n}(p) = 1 - (1-p^{n})^{m}$ for $m,n \ge 2$ each have exactly one stationary point in $(0,1)$ which is repelling. Both $0$ and $1$ are attracting fixed points. 
\end{proposition}

\begin{theorem}{\cite{HamblyKumagai}}\label{thm:pcdiamond}
The critical probability of percolation on $D(m,n)$ is the stationary point of $f_{m,n}(p)$ in $(0,1)$. 
\end{theorem}

\begin{proof}
From Proposition \ref{prop:level1} we know that the probability of crossing $D_1(m,n)$ with parameter $p$ is given by $f_{m,n}(p)$. When we next consider $D_2(m,n)$, each edge in $D_1(m,n)$ is replaced by a copy of $D_1(m,n)$ so the probability of crossing $D_2(m,n)$ is $f^{\circ 2}_{m,n}(p)$. Iterating we see that depending on the initial value of $p$, the probability of crossing $D_k(m,n)$ is $f^{\circ k}_{m,n}(p)$, which is limiting to the attracting stationary points of $f_{m,n}$, which are $0$ and $1$. So for $p < p_c(m,n)$, $f^{\circ k }_{m,n}(p)$ will limit to $0$, for $p>p_c(m,n)$ $f^{\circ k }_{m,n}(p)$ will limit to $1$, and for $p = p_c$, the limit will be $p_c>0$. So $p_c$ is the critical probability of percolation.
\end{proof}

For $m=n=4$ $p_c(4,4) \sim 0.282$. This is the first bound that is mentioned in proof of Theorem \ref{thm:main}.

\begin{proposition}\label{prop:pcmn}
Let $p_c(m,n)$ denote the critical probability of percolation on $D(m,n)$. Then for $m,n \ge 2$, $\frac{\partial p_c}{\partial m} < 0$ and $\frac{\partial p_c}{\partial n} > 0.$
\end{proposition}

\begin{proof}
Consider the function
\begin{equation*}
	g_{m,n}(p) =  f_{m,n}(p) -p = (1-p)-(1-p_c^{n})^{m}.
\end{equation*}
By Proposition \ref{prop:fmn}, $p_c$ is the unique zero of $g$ in $(0,1)$; also, $g_{m,n}(0) = 0$ and $g_{m,n}(1) = 0$. We will need three derivatives of $g_{m,n}$:
\begin{eqnarray}
	\frac{\partial g_{m,n}}{\partial p} &=& -1 + mnp^{n-1}(1-p^{n})^{m-1}\label{eq:gp}\\
	\frac{\partial g_{m,n}}{\partial m} &=& -n(1-p^{n})^{m}\ln(1-p^{n})np^{n-1}\label{eq:gm}\\
	\frac{\partial g_{m,n}}{\partial n} &=& -mp^{n}(1-p^{n})^{m-1}\ln(p).\label{eq:gn}
\end{eqnarray}
From (\ref{eq:gp}) we see that $\frac{\partial g_{m,n}}{\partial p}(0) < 0$ so on the interval $(0,p_c)$ $g_{m,n} <0$. Then by (\ref{eq:gm}) if $m$ is increasing so is $g_{m,n}(p)$ meaning that $p_c(m,n)$ is decreasing in $m$. Simiarly by (\ref{eq:gn}) $p_c(m,n)$ is increasing in $n$.
\end{proof}

This Proposition justifies the principle that to find the diamond with the lowest $p_c$ the goal is to find the optimal tradeoff between having the most branches with the fewest subdivisions. This principle will inform the bound in Theorem \ref{thm:main}.

\section{Barycentric Subdivisions and non-p.c.f. Sierpinski gasket}\label{sec:BSSG}

In this section we consider first the barycentric subdivision of a triangle and bond percolation on the limit of repeated barycentric subdivision (denoted $T$) in Theorem \ref{thm:pcT}. Then we define the non-p.c.f. Sierpinski gasket as it is presented in \cite{nonpcf} and extend the proof of Theorem \ref{thm:pcT} to non-p.c.f. Sierpinski gaskets ($S$) in Theorem \ref{thm:pSup}. The section ends with a comparison between the critical percolation probabilities on the $T$ and $S$ through an intermediate graph $\tilde{S}$ which has the same vertex set as $T$ but the edge set inherited from $S$. 

Throughout this section we will be concerned with estimating $p_G$ for the fractal limits of planar graphs. As in Section \ref{sec:diamond}, $p_G$ is interpreted as the supremum of values of $p$ such that the probability of crossing the $n^{th}$ level approximating graph tends to zero as $n$ increases. In Section \ref{sec:hex} when we consider the hexacarpet, the usual formulation of there being an infinite open cluster containing the origin is equivalent to the crossing probability formulation because the hexacarpet has uniformly bounded vertex degree and no vertex accumulation points.

\subsection{Barycentric Subdivision}\label{ssec:barycentric}

We now consider the iterated barycentric sudivision of a triangle. The purpose of this discussion is to show how this structure is related to the non-p.c.f. Sierpinski gasket. Throughout the following we will adopt some of the notation from \cite{Begue}.

\begin{figure}[t]
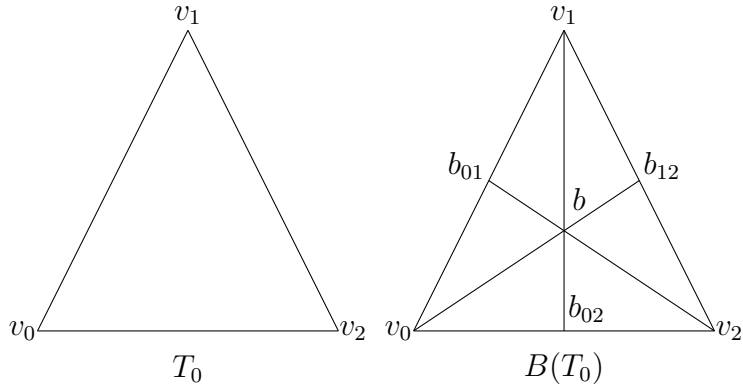

\begin{pgfpicture}{-0.5cm}{-0.5cm}{9.5cm}{4.5cm}
\pgfxyline(0,0)(2,4)
\pgfxyline(0,0)(4,0)
\pgfxyline(2,4)(4,0)

\pgfputat{\pgfxy(-0.2,0)}{\pgfbox[center,center]{$v_0$}}
\pgfputat{\pgfxy(2,4.2)}{\pgfbox[center,center]{$v_1$}}
\pgfputat{\pgfxy(4.2,0)}{\pgfbox[center,center]{$v_2$}}

\pgfxyline(5,0)(7,4)
\pgfxyline(5,0)(9,0)
\pgfxyline(7,4)(9,0)
\pgfxyline(6,2)(9,0)
\pgfxyline(7,0)(7,4)
\pgfxyline(5,0)(8,2)

\pgfputat{\pgfxy(4.8,0)}{\pgfbox[center,center]{$v_0$}}
\pgfputat{\pgfxy(7,4.2)}{\pgfbox[center,center]{$v_1$}}
\pgfputat{\pgfxy(9.2,0)}{\pgfbox[center,center]{$v_2$}}
\pgfputat{\pgfxy(5.7,2.2)}{\pgfbox[center,center]{$b_{01}$}}
\pgfputat{\pgfxy(8.3,2.2)}{\pgfbox[center,center]{$b_{12}$}}
\pgfputat{\pgfxy(7.3,0.3)}{\pgfbox[center,center]{$b_{02}$}}
\pgfputat{\pgfxy(7.2,1.75)}{\pgfbox[center,center]{$b$}}

\pgfputat{\pgfxy(2,-0.5)}{\pgfbox[center,center]{$T_0$}}
\pgfputat{\pgfxy(7,-0.5)}{\pgfbox[center,center]{$B(T_0)$}}

\end{pgfpicture}
\caption{\label{T0} Barycentric Subdivision of $T_0$}
\end{figure}

\begin{definition}
The graph $T_0$ is shown in Figure \ref{T0}. Let $b_{ij}, \ i,j = 0,1,2$ denote the midpoints of the edges of $T_0$ defined by $(v_i,v_j)$, and let $b = \frac{1}{3}(v_0+v_1+v_2)$ be the barycenter of $T_0$. The graph $T_1 = B(T_0)$, which is the barycentric subdivision of $T_0$, is also shown in Figure \ref{T0}. Since $T_1$ is a simplicial complex, it can itself be the object of a barycentric subdivision of each simplicial face.  Define $T_n$ as the result of barycentrically subdividing every face in $T_{n-1}$. Denote the limit as $n \rightarrow \infty$ by $T$. The limit is in the sense of an infinite planar graph.
\end{definition}

Notice that the vertex set of $T_1$ is the same as the vertex set of $S_1$ (see Figure \ref{fig:nonpcf}), the first approximation to the non-p.c.f. Sierpinski gasket, but there are no multi-edges. This is no longer true for $n \ge 2$, see Figure \ref{B2(T0)} for $B^{2}(T_0)$. This is the motivation for defining the graphs $\tilde{S}_n$ that are introduced below. The notion of barycentric subdivision is very well studied and we will refer to \cite{HatcherAT} for background.

\begin{figure}[t]
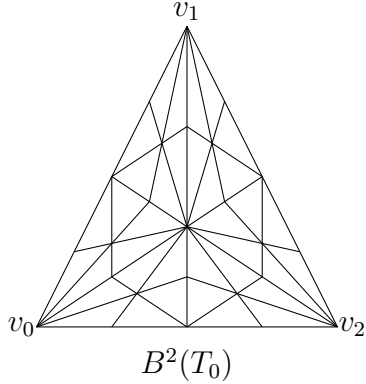

\begin{pgfpicture}{-0.5cm}{-0.5cm}{4.5cm}{4.5cm}
\pgfxyline(0,0)(2,4)
\pgfxyline(0,0)(4,0)
\pgfxyline(2,4)(4,0)
\pgfxyline(1,2)(4,0)
\pgfxyline(2,0)(2,4)
\pgfxyline(0,0)(3,2)

\pgfxyline(1,2)(1,0.666)
\pgfxyline(2,1.333)(0.5,1)
\pgfxyline(0,0)(1.5,1.666)

\pgfxyline(1.5,1.666)(2,4)
\pgfxyline(2,1.333)(1.5,3)
\pgfxyline(1,2)(2,2.666)

\pgfxyline(2,2.666)(3,2)
\pgfxyline(2,1.333)(2.5,3)
\pgfxyline(2,4)(2.5,1.666)

\pgfxyline(2.5,1.666)(4,0)
\pgfxyline(2,1.333)(3.5,1)
\pgfxyline(3,2)(3,0.666)

\pgfxyline(3,0.666)(2,0)
\pgfxyline(2,1.333)(3,0)
\pgfxyline(4,0)(2,0.666)

\pgfxyline(2,0.666)(0,0)
\pgfxyline(2,1.333)(1,0)
\pgfxyline(2,0)(1,0.666)

\pgfputat{\pgfxy(-0.2,0)}{\pgfbox[center,center]{$v_0$}}
\pgfputat{\pgfxy(2,4.2)}{\pgfbox[center,center]{$v_1$}}
\pgfputat{\pgfxy(4.2,0)}{\pgfbox[center,center]{$v_2$}}

\pgfputat{\pgfxy(2,-0.5)}{\pgfbox[center,center]{$B^2(T_0)$}}
\end{pgfpicture}
\caption{\label{B2(T0)} Repeated Barycentric Subdivision of $T_0$}
\end{figure}

\begin{remark}\label{rem:barysides}
Two properties of barycentric subdivision of triangles are going to be particularly useful. The first is that the area of each triangular face is $6^{-n}$ of the area of $T_0$ \cite{Diaconis}. The second is that the mesh size at level $n$ (the length of the longest edge in $B^{n}(T_0)$ is $\frac{2}{3}$ of the mesh size at level $n-1$ \cite{HatcherAT}. If $T_0$ is an equialateral triangle with side length one then the longest edge in the $B^{n}(T_0)$ is $(\frac{2}{3})^{n}$. 
\end{remark}

The following notation will be useful in the proof of Theorem \ref{thm:pcT}. Let $T_k = B^{k}(T_0)$. We label each of the six sub-triangles in $T_1$ in a clockwise direction: $0 = \{v_0,b_{01},b\}$, $1 = \{b_{01},v_1,b\}$, etc. For each sub triangle there are three vertices, three mid-points, and the barycenter, all but the barycenter might be shared with another face or more. We label each of these points recursively where $i(v_0)$ is the lower left vertex of $i$, $i(b)$ is the barycenter of $i$, and so forth. We label the sub-triangles in $T_2$ in a similar manner. For $i$, a sub-triangle of $T_1$, we have $i0 = \{i(v_0),i(b_{01}),i(b)\}$ and label the other five sub-triangles of $i$ in a clockwise manner. 

\begin{figure}[t]
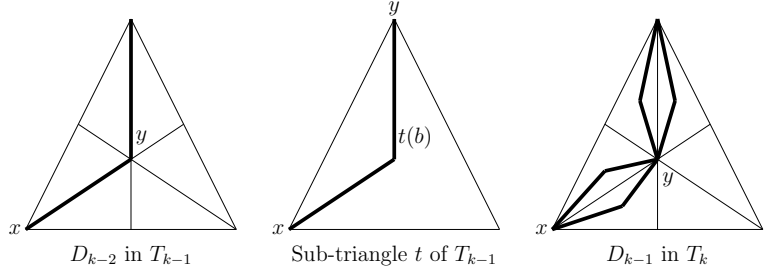

\scalebox{.7}{
\begin{pgfpicture}{-0.2cm}{-0.5cm}{14cm}{4.5cm}
\pgfxyline(0,0)(2,4)
\pgfxyline(0,0)(4,0)
\pgfxyline(2,4)(4,0)
\pgfxyline(1,2)(4,0)
\pgfxyline(2,0)(2,4)
\pgfxyline(0,0)(3,2)

\pgfxyline(5,0)(7,4)
\pgfxyline(5,0)(9,0)
\pgfxyline(7,4)(9,0)

\pgfxyline(10,0)(12,4)
\pgfxyline(10,0)(14,0)
\pgfxyline(12,4)(14,0)
\pgfxyline(11,2)(14,0)
\pgfxyline(12,0)(12,4)
\pgfxyline(10,0)(13,2)

\begin{pgfscope}
\pgfsetlinewidth{2pt}

\pgfxyline(0,0)(2,1.333)
\pgfxyline(2,4)(2,1.333)

\pgfxyline(5,0)(7,1.333)
\pgfxyline(7,4)(7,1.333)

\pgfxyline(10,0)(11,1.111)
\pgfxyline(12,1.333)(11,1.111)

\pgfxyline(10,0)(11.333,0.444)
\pgfxyline(12,1.333)(11.333,0.444)

\pgfxyline(12,1.333)(11.666,2.444)
\pgfxyline(12,4)(11.666,2.444)

\pgfxyline(12,1.333)(12.333,2.444)
\pgfxyline(12,4)(12.333,2.444)
\end{pgfscope}

\pgfputat{\pgfxy(4.8,0)}{\pgfbox[center,center]{$x$}}
\pgfputat{\pgfxy(7,4.2)}{\pgfbox[center,center]{$y$}}
\pgfputat{\pgfxy(7.4,1.75)}{\pgfbox[center,center]{$t(b)$}}

\pgfputat{\pgfxy(-0.2,0)}{\pgfbox[center,center]{$x$}}
\pgfputat{\pgfxy(2.2,1.75)}{\pgfbox[center,center]{$y$}}

\pgfputat{\pgfxy(9.8,0)}{\pgfbox[center,center]{$x$}}
\pgfputat{\pgfxy(12.2,0.95)}{\pgfbox[center,center]{$y$}}

\pgfputat{\pgfxy(2,-0.5)}{\pgfbox[center,center]{$D_{k-2}$ in $T_{k-1}$}}
\pgfputat{\pgfxy(7,-0.5)}{\pgfbox[center,center]{Sub-triangle $t$ of $T_{k-1}$}}
\pgfputat{\pgfxy(12,-0.5)}{\pgfbox[center,center]{$D_{k-1}$ in $T_k$}}
\end{pgfpicture}
}
\caption{\label{embedding} Embedding $D_{k-1}$ into $T_k$}
\end{figure}

\begin{theorem}\label{thm:pcT}
Let $T$ be the limit of the barycentric subdivisions of the triangle and let $p_T$ be the critical probability of percolation on $T$. Then $p_T < 0.282.$
\end{theorem}

\begin{proof}
It is sufficient to show an embedding of $D_{k-1}$ into $T_k$ which then implies that $D(2,2) \subset T$. Then by the monotonicity of the critical probability in the subgraph relation $p_T \le p_c(2,2)$. We begin by embedding $D_1$ into $T_2$. Note that there are in fact three copies of $D_1$ in $T_2$: $\{v_0,1(b),b,6(b)\}$, $\{v_1,2(b),b,3(b)\}$, and $\{v_2,4(b),b,5(b)\}$. Now suppose that we have an embedding of $D_{k-2}$ in $T_{k-1}$. The rule for identifying the embedding of $D_{k-1}$ in $T_k$ is as follows: for each lowest-level sub-triangle $t$ of $T_{k-1}$ containing an edge $\{x,y\}$ of the embedded $D_{k-2}$, the edge $\{x,y\}$ is replaced by two edges $\{x,t(b)\}$ and $\{t(b),y\}$ (see Figure \ref{embedding}). However since each edge $\{x,y\}$ is an edge of two sub-triangles, this actually replaces the edge $\{x,y\}$ with a copy of $D_1(2,2)$, which is the recursive definition of $D_k(2,2)$. 

Three copies of $D$ can be embedded in $T$, see Figure \ref{embedding} for where two of them can be taken. Because there are only finitely many diamonds needed to cross $T$
, we have that $\frac{\sqrt{5}-1}{2} = p_c(2,2) \ge p_T$. 

By a similar construction $D(2^{j},2^{n})$ for $j \ge n$ could also be embedded. These constructions are notationally inelegant and so are omitted. However by using Proposition \ref{prop:pcmn} it can be shown that $p_c(4,4)$ is the minimum critical value amongst these families and that $0.281< p_c(4,4) < 0.282$.
\end{proof}

\subsection{The non-p.c.f. Sierpinski Gasket}\label{ssec:nonSG}

\begin{figure}[t]
\includegraphics[scale=1]{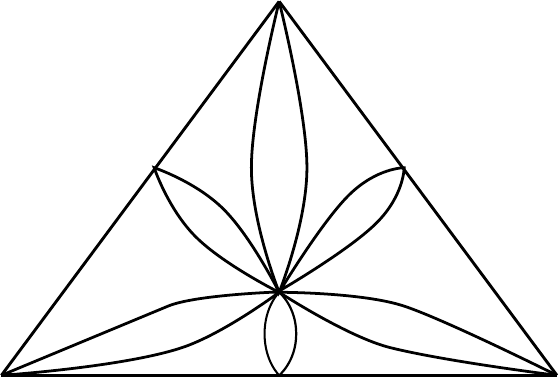} 
\caption{Level one approximation the non-p.c.f. Sierpinski gasket.}
\label{fig:nonpcf}
\end{figure}

As was discussed in the introduction \cite{Kigami} presents the full definition of post-critically finite fractals. However a sufficient criterion for a fractal to not be p.c.f. is for the approximating graphs to have unbounded vertex degree. The non-p.c.f. Sierpinski gasket was introduced in \cite{nonpcf} as an example of finitely ramified fractal which is not p.c.f., in \cite{BajorinEtAl} the spectra of Laplacians on it  and the diamond fractal were calculated. There is a sense in which the diamond fractal can be seen as a non-p.c.f. analog to the unit interval when it is viewed as a p.c.f. set.  In this same sense the non-p.c.f. Sierpinski gasket is an analog of the usual Sierpinski gasket. When viewed as a p.c.f. set, the unit interval is approximated by graphs constructed by taking an edge and adding a number of nodes to subdivide the pre-existing edge, that is one branch but multiple subdivisions. The diamond fractal comes about where in addition to this procedure extra copies are added. In the same way the non-p.c.f. Sierpinski gasket is a way to add extra copies to the approximating graphs for the Sierpinski gasket. The extra vertex at the barycenter is introduced this way.  

\begin{definition}
Let $S_0$ be a the complete graph on $\{v_0,v_1,v_2\}$. Let $S_1$ be the graph with vertices $\{v_0,v_1,v_2,b_{01},b_{12}b_{20},b\}$ and edges as in Figure \ref{fig:nonpcf}. We will say that $S_1$ consists of six triangles whose edge sets are disjoint. That is the triangles $\{v_0,b_{01},b\}$ and $\{v_1,b_{01},b\}$ each use a different one of the multiple edges connecting $b_{01}$ and $b$. Iteratively define $S_n$ by taking $S_{n-1}$ and replacing each triangle with a copy of $S_1$. Denote by $S$ the limiting infinite graph, which we call the non-p.c.f. Sierpinski gasket.
\end{definition}

\begin{theorem}\label{thm:pSup}
Let $S$ be the non-p.c.f. Sierpinski gasket and $p_S$ be the critical probability of percolation on $S$. Then $p_S < 0.282$.
\end{theorem}

\begin{proof} 
The proof of Theorem \ref{thm:pcT} goes through for $S$ using diamonds with two or four branches but not for eight or more branches. This limit comes from the structure of $S_{2}$ and $S_{3}$ where the number of non-intersecting paths from $v_0$ to $b$ not going through $b_{01}$ or $b_{20}$ achieves its maximum of $4$. The labels for the points are as in Figure \ref{T0}. Hence a maximum of four branches.  However, this is sufficient for the same upper bound as for $T$. 
\end{proof}

\subsection{Equivalence of Percolation on T and S}\label{ssec:equiv}
We now turn to the relationship between $p_T$ and $p_S$. 

\begin{definition}
Let $T_0 = (x,y,z)$ be a spacial triangle, that is a particular embedding of $T_0$ into $\mathbb{R}^{2}$. Then the boundary of $T_0$ is the union of the edges, that is 
\begin{equation*}
	\partial T_0 = (x,y) \cup (y,z) \cup (z,x),
\end{equation*}
and the interior of $T_0$ is $\stackrel{\circ}{T_0} = T_0 \setminus \partial T_0$.
\end{definition}

\begin{definition}\label{def:collapse}
Let $\{S_n\}_{n=0}^{\infty}$ be the sequence of graph approximations to $S$ and $\omega_S(e)$ to be a labeling of the edges in $S$. Notice that $S_n$ is a graph minor of $S_{n+1}$ so there is a natural way in which the vertices of $S_n$ are also vertices of $S_{n+1}$. Identify the vertices $x,z$ in $S$ such that for some $n \ge 2$, $x,z \in S_n$ and they are mid-points of a pair of multi edges in $S_{n-1}$. The quotient by this equivalence relation maps the vertices of $S$ to the vertices of $T$. All adjacent pairs of vertices are joined by one or two edges, call this graph $\tilde{S}$. If there are two edges $e_1,e_2$ replace them by a single edge $e$ and assign $\omega_T(e) = \min\{\omega_S(e_1),\omega_S(e_2)\}$.

This procedure maps a labeling on the edges of $S$ to a labeling on the edges of $T$, denote the map by $\pi$. When writing $\pi(A)$ for $A \subset \Omega_S$ we mean that $\pi$ is applied to each element of $A$. 
\end{definition}

The action of $\pi$ acts naturally in two steps. The first is to take $S$ and map the vertex set of $S$ onto the vertex set of $T$. The second is to map the edge set of $S$ onto the edge set of $T$. The graph obtained between these two steps is called $\tilde{S}$. Notice also that $T$ is the simple graph resulting from replacing all multi-edges in $\tilde{S}$ with a single edge. 

As an immediate consequence of Lemma \ref{lem:mono2} we get 

\begin{corollary}
$p_{\tilde{S}} \le p_S < 0.282$. 
\end{corollary}

The following theorem shows the advantage of modeling open and closed edges based on a random variable taking values in $[0,1]$ rather than letting an edge be open or closed as a Bernoulli trial with parameter $p$. 

\begin{theorem}\label{thm:SBcomp}
$p_T/2 \le p_{\tilde{S}} \le p_T$
\end{theorem}

\begin{proof}
Since $T$ can be seen as a subgraph of $\tilde{S}$ with the edges chosen so that if two edges connect a pair of vertices the edge with the smaller value of $\omega(e)$ is chosen. We have by Remark \ref{rem:mono1} that $p_{\tilde{S}} \le p_T.$

The edge set, $E_{\tilde{S}}$, can be partitioned so that in each subset are edges that connect the same two vertices. In each element of this partition is either one or two edges. The edge set of $T$ is the set of partitions of $E_{\tilde{S}}$. So the event: 
\begin{quote}
	$\{$the set of vertices in $\tilde{S}$ connected to the origin by elements of the partition where in each partition there is at least one open edge with parameter $p$ is infinite$\}$
\end{quote}
 can be rewritten as 
 \begin{quote}
 	$\{$there is an infinite open cluster in $T$ containing the origin with parameter $p$ on $T_0$ and parameter $p(2-p)$ on $\stackrel{\circ}{T_0}$.$\}$
\end{quote}
This second event is contained in the event that 
\begin{quote}
	$\{$with parameter $2p$ there is an infinite open cluster containing the origin in $T$.$\}$
\end{quote}
Thus 
\begin{equation*}
	\theta_T(2p) \ge \theta_{\tilde{S}}(p).
\end{equation*}
This then implies that $p_T/2 \le p_{\tilde{S}}$.
\end{proof}

\begin{corollary}\label{cor:SScomp}
$p_S \ge p_T/2$.
\end{corollary}

\section{The Hexacarpet}\label{sec:hex}
The original impetus for studying the hexacarpet was that it was suspected to be computationally more tractable than the Sierpinski carpet for a summer REU project at Cornell University. The intuition was that the result would be infinitely ramified but the number of vertices in the approximating graphs would grow slower than for the Sierpinski carpet. However, as computers became more powerful over the last 20 years it was possible to computationally study the Sierpinski carpet itself and the hexacarpet was passed over for the theoretically better understood Sierpinski carpet \cite{TeplyaevPrivate}. Begue et Al. \cite{Begue} returned to the study of the hexacarpet in 2011. 

\begin{definition}\label{def:hexa}
Let $T_n = B^{n}(T_0)$ be the $n^{th}$ barycentric subdivision of the triangle $T_0$. Set $H_n$ to be the planar pre-dual of of $T_n$ (see Figure \ref{dual1}). Let $H =\lim_{n \rightarrow \infty} H_n$. Call $H$ the hexacarpet.
\end{definition} 

Planar duals are discussed in more detail in Definition \ref{dualdef} in the next section.

\begin{figure}[t]
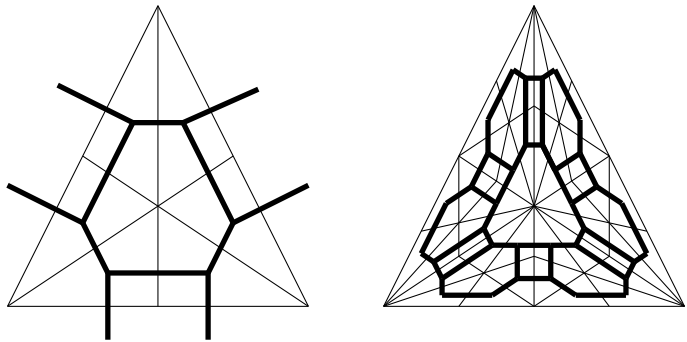

\begin{pgfpicture}{-5cm}{-0.25cm}{4cm}{4cm}

\pgfxyline(0,0)(2,4)
\pgfxyline(0,0)(4,0)
\pgfxyline(2,4)(4,0)
\pgfxyline(1,2)(4,0)
\pgfxyline(2,0)(2,4)
\pgfxyline(0,0)(3,2)

\pgfxyline(1,2)(1,0.666)
\pgfxyline(2,1.333)(0.5,1)
\pgfxyline(0,0)(1.5,1.666)

\pgfxyline(1.5,1.666)(2,4)
\pgfxyline(2,1.333)(1.5,3)
\pgfxyline(1,2)(2,2.666)

\pgfxyline(2,2.666)(3,2)
\pgfxyline(2,1.333)(2.5,3)
\pgfxyline(2,4)(2.5,1.666)

\pgfxyline(2.5,1.666)(4,0)
\pgfxyline(2,1.333)(3.5,1)
\pgfxyline(3,2)(3,0.666)

\pgfxyline(3,0.666)(2,0)
\pgfxyline(2,1.333)(3,0)
\pgfxyline(4,0)(2,0.666)

\pgfxyline(2,0.666)(0,0)
\pgfxyline(2,1.333)(1,0)
\pgfxyline(2,0)(1,0.666)

\begin{pgfscope}
\pgfsetlinewidth{2pt}

\pgfxyline(.5,0.704)(0.833,1.370)
\pgfxyline(0.833,1.370)(1.166,1.592)
\pgfxyline(1.166,1.592)(1.5,1.37)
\pgfxyline(1.5,1.37)(1.333,1.036)
\pgfxyline(1.333,1.036)(0.666,0.592)
\pgfxyline(0.666,0.592)(.5,.704)

\pgfxyline(1.166,1.592)(1.388,2.036)
\pgfxyline(1.5,1.37)(1.722,1.814)

\pgfxyline(1.722,1.814)(1.388,2.036)
\pgfxyline(1.388,2.036)(1.388,2.481)
\pgfxyline(1.388,2.481)(1.722,3.148)
\pgfxyline(1.722,3.148)(1.888,3.036)
\pgfxyline(1.888,3.036)(1.888,2.147)
\pgfxyline(1.888,2.147)(1.722,1.814)

\pgfxyline(1.888,2.147)(2.111,2.147)
\pgfxyline(1.888,3.036)(2.111,3.036)

\pgfxyline(2.111,3.036)(2.111,2.147)
\pgfxyline(2.111,3.036)(2.277,3.148)
\pgfxyline(2.277,3.148)(2.611,2.481)
\pgfxyline(2.611,2.481)(2.611,2.036)
\pgfxyline(2.611,2.036)(2.277,1.814)
\pgfxyline(2.277,1.814)(2.111,2.147)

\pgfxyline(2.111,2.147)(2.5,1.37)
\pgfxyline(2.611,2.036)(2.833,1.592)

\pgfxyline(2.5,1.37)(2.833,1.592)
\pgfxyline(2.833,1.592)(3.166,1.37)
\pgfxyline(3.166,1.37)(3.5,0.703)
\pgfxyline(3.5,0.703)(3.333,0.592)
\pgfxyline(3.333,0.592)(2.666,1.036)
\pgfxyline(2.666,1.036)(2.5,1.37)

\pgfxyline(3.333,.592)(3.222,.37)
\pgfxyline(2.666,1.036)(2.555,.814)

\pgfxyline(3.222,.37)(2.555,.814)
\pgfxyline(3.222,.37)(3.222,.148)
\pgfxyline(3.222,.148)(2.555,.148)
\pgfxyline(2.555,.148)(2.222,.37)
\pgfxyline(2.222,.37)(2.222,.814)
\pgfxyline(2.222,.812)(2.555,.814)

\pgfxyline(2.222,.812)(1.777,.812)
\pgfxyline(2.222,.37)(1.777,.37)

\pgfxyline(1.777,.812)(1.777,.37)
\pgfxyline(1.777,.37)(1.444,.148)
\pgfxyline(1.444,.148)(.777,.148)
\pgfxyline(.777,.148)(.777,.37)
\pgfxyline(.777,.37)(1.444,.814)
\pgfxyline(1.444,.814)(1.777,.812)

\pgfxyline(.777,.37)(0.666,0.592)
\pgfxyline(1.444,.814)(1.333,1.036)

\end{pgfscope}

\pgfxyline(-5,0)(-3,4)
\pgfxyline(-5,0)(-1,0)
\pgfxyline(-3,4)(-1,0)
\pgfxyline(-4,2)(-1,0)
\pgfxyline(-3,0)(-3,4)
\pgfxyline(-5,0)(-2,2)

\begin{pgfscope}
\pgfsetlinewidth{2pt}

\pgfxyline(-3.666,0.444)(-4,1.111)
\pgfxyline(-4,1.111)(-3.333,2.444)
\pgfxyline(-3.333,2.444)(-2.666,2.444)
\pgfxyline(-2.666,2.444)(-2,1.111)
\pgfxyline(-2,1.111)(-2.333,0.444)
\pgfxyline(-2.333,0.444)(-3.666,0.444)

\pgfxyline(-3.666,0.444)(-3.666,-0.444)
\pgfxyline(-2.333,0.444)(-2.333,-0.444)
\pgfxyline(-4,1.111)(-5,1.611)
\pgfxyline(-4.333,2.943)(-3.333,2.444)
\pgfxyline(-2,1.111)(-1,1.611)
\pgfxyline(-2.666,2.444)(-1.666,2.889)
\end{pgfscope}

\end{pgfpicture}
\caption{\label{dual1}Dual Graphs}
\end{figure}

The goal of this section is to show that $p_H<1$. Observe from Figures \ref{dual1} and \ref{Hexagon} that $H_1$ can be embedded in $\mathbb{R}^{2}$ with edge lengths at least $1$. This minimal edge length will be used later. Another observation to be taken from these two pictures is that $H_1$ is composed of six copies of $H_0$ joined together. Similarly, six copies of $H_1$ could be joined together to form $H_2$ and with the appropriate scaling, all of the edges of $H_2$ would have length at least $1$. The increasing union of $H_n$ embedded in this manner is the hexacarpet $H$. The following theorem of Kozma is based on relating cut sets in the primal graph with open box crossings in the dual graph using a Pierl-type path counting argument.

\begin{theorem}{(Kozma, \cite{Kozma})}\label{thm:kozma}
Let $G$ be a planar graph with no vertex accumulation points such that
\begin{enumerate}
	\item There exists numbers $K$ and $D$ such that for all $v \in G$ and for all $r \ge 1$ one has for the open ball in the Euclidean metric, $B(v,r)$, that the number of vertices in the ball satisfies $|B(v,r)| \le Kr^{D}$. 
	\item There exists numbers $k$ and $\epsilon >0$ such that for any finite non-empty set of vertices $J \subset G$, $|\partial J| \ge k|J|^{\epsilon}$.
\end{enumerate}
Let $p_G$ be the critical probability for independent bond percolation on $G$. Then $p_G < 1$. 
\end{theorem}

\begin{figure}[t]
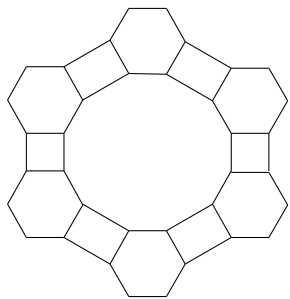

\begin{pgfpicture}{0cm}{-0.8cm}{4cm}{3.1cm}

\pgfxyline(0,0.443)(0.25,0)
\pgfxyline(0.25,0)(0.75,0)
\pgfxyline(0.75,0)(1,0.433)
\pgfxyline(1,0.433)(0.75,0.886)
\pgfxyline(0.75,0.886)(0.25,0.886)
\pgfxyline(0.25,0.886)(0,0.443)

\pgfxyline(0.25,0.886)(0.25,1.386)
\pgfxyline(0.75,0.886)(0.75,1.386)

\pgfxyline(0,1.829)(0.25,1.386)
\pgfxyline(0.25,1.386)(0.75,1.386)
\pgfxyline(0.75,1.386)(1,1.829)
\pgfxyline(1,1.829)(0.75,2.272)
\pgfxyline(0.75,2.272)(0.25,2.272)
\pgfxyline(0.25,2.272)(0,1.829)

\pgfxyline(0.75,0)(1.362,-0.353)
\pgfxyline(1,0.443)(1.612,0.09)

\pgfxyline(1.612,0.09)(1.362,-0.353)
\pgfxyline(1.362,-0.353)(1.612,-0.786)
\pgfxyline(1.612,-0.786)(2.112,-0.786)
\pgfxyline(2.112,-0.786)(2.362,-0.353)
\pgfxyline(2.362,-0.353)(2.112,0.09)
\pgfxyline(2.112,0.09)(1.612,0.09)

\pgfxyline(2.362,-0.353)(2.974,0)
\pgfxyline(2.112,0.09)(2.724,0.443)

\pgfxyline(2.974,0)(2.724,0.443)
\pgfxyline(2.974,0)(3.474,0)
\pgfxyline(3.474,0)(3.724,0.433)
\pgfxyline(3.724,0.433)(3.474,0.866)
\pgfxyline(3.474,0.866)(2.974,0.866)
\pgfxyline(2.974,0.886)(2.724,0.443)

\pgfxyline(2.974,0.886)(2.974,1.386)
\pgfxyline(3.474,0.886)(3.474,1.386)

\pgfxyline(2.974,1.386)(3.474,1.386)
\pgfxyline(2.974,1.386)(2.724,1.819)
\pgfxyline(2.724,1.819)(2.974,2.252)
\pgfxyline(2.974,2.252)(3.474,2.252)
\pgfxyline(3.474,2.252)(3.724,1.819)
\pgfxyline(3.724,1.819)(3.474,1.386)

\pgfxyline(2.724,1.819)(2.112,2.172)
\pgfxyline(2.974,2.252)(2.362,2.605)
\pgfxyline(1,1.829)(1.612,2.182)
\pgfxyline(0.75,2.272)(1.362,2.625)

\pgfxyline(1.362,2.625)(1.612,2.182)
\pgfxyline(1.612,2.182)(2.112,2.172)
\pgfxyline(2.112,2.172)(2.362,2.605)
\pgfxyline(2.362,2.605)(2.112,3.048)
\pgfxyline(1.362,2.625)(1.612,3.048)
\pgfxyline(1.612,3.048)(2.112,3.048)

\end{pgfpicture}
\caption{\label{Hexagon} First Graph Approximation of the Hexacarpet}
\end{figure}

\begin{lemma}
The hexacarpet, $H$, satisfies Condition 1 of Theorem \ref{thm:kozma} and is a planar graph with no vertex accumulation points.
\end{lemma}

\begin{proof}
Let $H$ be the infinite graph limit of $H_n$. Since the $H_n$ are planar and $H_{n+1}$ consists of six copies of $H_n$ arranged in a ring and joined together, $H$ can be realized as an increasing union of $H_n$. All edges can be taken to have length at least one. So $H$ can be embedded as a planar graph with no vertex accumulation points. Take open balls of fixed positive radius, $\epsilon$ centered at each vertex. For $\epsilon$ small enough these balls are disjoint. Then a large ball $B(x,r) \subset \mathbb{R}^{2}$ can only contain quadratically many of these smaller balls. An upper bound on $K$ is given by the optimal packing density of disks in the plane whose exact value is not necessary, merely its existence. 
\end{proof}

\begin{lemma}
The hexacarpet, $H$, satisfies Condition 2 of Theorem \ref{thm:kozma}.
\end{lemma}

\begin{proof}
Let $J$ be a finite subset of vertices in $H$ and $\partial J$ the set of edges with one end in $J$ and the other not. 
The aim is to find $k$ and $\epsilon$ positive so that there is a positive lower bound on $\frac{|\partial J|}{k|J|^{\epsilon}}$ over all finite vertex sets, $J$.

 Suppose that $J$ is a connected subset of $H$ with cardinality between $6^{n-2}$ and $6^{n-1}$. Without loss of generality we can assume that $J$ is contained in some copy of $H_n$ as a subset of $H$. 

A set of vertices in $H_n$ corresponds to a collection of faces in $B^{n}(T_0)$. This temporarily changes the embedding of $H_n$ from one where all edges are length at least one described above to the embedding inherited from the definition of $H_n$ as a pre-dual graph of a triangulation of an equilateral triangle. This change of embedding does not upset the conclusions of the previous lemma since the conclusion of this lemma is statement about the graph structure itself and only uses a particular embedding to access the structure that is already present.  Let $\mathcal{H}$ be the collection of faces corresponding to $J$. It has total area $\frac{\sqrt{3}}{4}6^{-n}|J|$. Notice that $\mathcal{H}$ is a connected planar region and so cannot have a perimeter less than the circumference of a circle with the same area. Thus
\begin{equation*}
	perimeter(\mathcal{H}) \ge 2\sqrt{\pi}{\sqrt{\frac{\sqrt{3}}{4} 6^{-n}|J|}}.
\end{equation*}
Since $\mathcal{H}$ is a polygonal region whose edges have length less than $(\frac{2}{3})^{n}$ we can put a lower bound on the number of edges in the boundary of $\mathcal{H}$ which is $|\partial J|$,
\begin{eqnarray*}
	\left(\frac{2}{3}\right)^{n}|\partial J| & \ge & perimeter(\mathcal{H})\\
	|\partial J| &\ge& 2\sqrt{\pi}{\sqrt{\frac{\sqrt{3}}{4} 6^{-n}|J|}}\left( \frac{3}{2} \right)^{n}\\
	&=& C|J|^{1/2}\left( \frac{\sqrt{3}}{2\sqrt{2}} \right)^{n}\\
	&\ge & C'|J|^{1/2 + \alpha} \sim C'|J|^{0.7737}
\end{eqnarray*}
where 
$$\alpha  = \log_6\left(\frac{\sqrt{3}}{2\sqrt{2}}\right) \hspace{1cm} \text{and} \hspace{1cm} C' = \frac{3\sqrt[4]{3}}{8}\sqrt{\pi}.$$ 
It is crucial in obtaining these estimates that we could control $|J|$ by the area of $\mathcal{H}$. Then for Theorem \ref{thm:kozma} we choose $k = C'$ and $\epsilon = \frac{1}{2} + \alpha$.

If $J$ is not connected then its components can be connected by adding paths to $J$ connecting each component. Such paths can add no more edges to the boundary than vertices to $J$. Since $\epsilon < 1$, connecting $J$ increases $\frac{|\partial J|}{k|J|^{\epsilon}}$ so the estimate in Theorem \ref{thm:kozma}(2) holds for all finite sets $J$.
\end{proof}

\begin{corollary}\label{cor:pHup}
Following the above two lemmas and Theorem \ref{thm:kozma} $p_H < 1$.
\end{corollary}

\section{Hexacarpet and non-p.c.f. Sierpinski Gasket Duality}\label{sec:dual}
We now turn our attention to the relationship between percolation on the hexacarpet and the barycentric subdivisions of a triangle. We remarked earlier that the hexacarpet is the dual of the infinitely repeated barycentric subdivision of a triangle. That is, the dual graph of $T$ is $H$ with the vertex at infinity and edges incident to it deleted. 

\begin{definition}\label{dualdef} Let $G$ be a planar graph drawn in the plane. The (planar) dual of $G$, called $G_d$ is constructed by placing in each face of $G$ (including the infinite face if it exists) a vertex; for each edge $e$ of $G$, we place a corresponding edge joining the two vertices of $G_d$ which lie in the two faces of $G$ abutting the edge $e$.
\end{definition}

It remains to show how bond percolation on $T$ relates to bond percolation on $H$. We do so through the use of the following theorem from Bollob\'as and Riordan \cite{dualpc}. We will first need the following definition.

\begin{definition} A lattice $G$ has $k$-fold symmetry if the rotation about the origin through an angle of $2\pi/k$ maps the plane graph $G$ to itself and there is a group of translations that also map $G$ to itself.
\end{definition}

\begin{theorem}[Bollob\'as and Riordan \cite{dualpc}]\label{Bollobas} 
Let $G$ be a planar lattice with $k$-fold symmetry, $k\ge 2$, and let $G_D$ be its dual. Then $p_G+p_{G_D}=1$ for bond percolation.
\end{theorem}

There is a complication in applying this theorem in that $H$ and $T$ are not lattices. They are infinite graphs which can be used to tile the plane using only finitely many copies arranged with radial symmetry around the origin. Thus we have to consider approximating graphs which can be used to tile the plane to form a lattice.

\begin{theorem}\label{duality} 
Let $T$ and $H$ be as before. Then $p_T+p_H=1$.
\end{theorem}

\begin{proof}
Let $T'$ be the union of six unbounded copies of $T$ arranged with $C_6$ radial symmetry about the origin. Let $T'_n$ be the triangular tiling of the plane with the graphs $T_n$. Then $H'_n$ is a triangular tiling of the plane with graphs $H_n$ and also the planar dual of $T'_n$. Both $T'_n$ and $H'_n$ are lattices with $6-$fold symmetry. By Theorem \ref{Bollobas} $p_{T'_n} + p_{H'_n} = 1$ for all $n \ge 1$. It remains to be shown that $p_{T'_n} \rightarrow p_T$ and $p_{H'_n} \rightarrow p_H$. 

As $n \rightarrow \infty$, $T'_n \rightarrow T'$ and $H'_n$ approaches a similar arrangement of six copies of $H$ placed radially about the origin. Note that in this embedding $H$ has vertex accumulation points as a subset of $\mathbb{R}^{2}$ but not in the graph distance metric. 

Fix $\epsilon > 0$ and consider the event $\{ |\mathcal{O}_{H'}| > m\}.$ The origin here is to be taken as any vertex on the face of $H'$ containing the origin of $\mathbb{R}^{2}$. Since in the graph metric a local patch of $H$ is isometric to a patch in $H_n$ for large enough $n$ there exists an $N = N(m) > 0$ so that 
\begin{equation*}
	\mathbb{P}( |\mathcal{O}_{H'}| > m)  = \mathbb{P}_n ( |\mathcal{O}_{H_n'}| > m).
\end{equation*}
Where $\mathbb{P}_n$ is the product probability measure on $[0,1]^{E_{H_n'}}$. Equality can hold despite the different measures since the neighborhood of the origin for large enough $n$ in $H_{n}'$ and in $H'$ are isometric and the marginals of $\mathbb{P}_n$ and $\mathbb{P}$ coincide if the event only concerns those edges. Now choose $M =M(\epsilon)$ large enough so that for $m \ge M$
\begin{equation*}
	\mathbb{P}( |\mathcal{O}_{H'}| > m) \in [\theta_{H'}(p),\theta_{H'}(p)+\epsilon].
\end{equation*}
This means that for all $n \ge N(M(\epsilon))$ and $m \ge M(\epsilon)$, 
\begin{equation*}
	\mathbb{P}_n( |\mathcal{O}_{H_n'}| > m) \in [\theta_{H'}(p),\theta_{H'}(p)+\epsilon].
\end{equation*}
So the limits over $m$ and $n$ can be taken in either order. Taking the limit over $m$ first we get $\lim_{n \rightarrow} p_{H_n'} = p_{H'}$.

For $T$ we must recall that every vertex in $T$ has infinite degree so that the $T_n$ do not share a common patch around the origin. But as with the diamond fractal, $p_T$ is the limit of $p_{T_n}$, the crossing probabilities for the finite approximating graphs to $T$. What remains is to show that $\lim_{n \rightarrow \infty} p_{T_n'} - p_{T_n} = 0$. Since $T'_n$ is a tiling of $T_n$ we can mark the origin as the center of some copy of $T_n$. Then $T'_n$ and $T_n$ share a common graph isometric neighborhood around the origin and the argument used for the hexacarpets applies again.Thus $p_{T'_n}$ and $p_{T_n}$ have a common limit, $p_T$.
\end{proof}

\bibliographystyle{plain}
\bibliography{Bib/bib}{}

\end{document}